\documentclass[12pt,a4paper]{article}
\usepackage{amsthm}
\usepackage{amsmath}
\usepackage{epsfig}
\usepackage{color} 
\usepackage{amssymb}

\newtheorem{theorem}{Theorem}[section]

\newtheorem{corollary}[theorem]{Corollary}

\newtheorem*{maintheorem}{Theorem}

\newcommand{\supp}{\operatorname{supp}}

\begin{document}

\title{\large{\textbf{CHARACTERISTIC FLOWS ON SIGNED GRAPHS\\
                      AND SHORT CIRCUIT COVERS}}}

\author{
Edita M\'a\v cajov\' a
and Martin \v{S}koviera\\[3mm]
Department of Computer Science\\
Faculty of Mathematics, Physics and Informatics\\
Comenius University\\
842 48 Bratislava, Slovakia\\[2mm]
{\small\tt macajova@dcs.fmph.uniba.sk}\\[-1mm]
{\small\tt skoviera@dcs.fmph.uniba.sk}}

\date{}

\maketitle

\begin{abstract}
We generalise to signed graphs a classical result of Tutte
[Canad. J. Math. 8 (1956), 13--28] stating that every integer
flow can be expressed as a sum of characteristic flows of
circuits. In our generalisation, the r\^ole of circuits is
taken over by signed circuits of a signed graph which occur in
two types -- either balanced circuits or pairs of disjoint
unbalanced circuits connected with a path intersecting them
only at its ends. As an application of this result we show that
a signed graph $G$ admitting a nowhere-zero $k$-flow has a
covering with signed circuits of total length at most
$2(k-1)|E(G)|$.
\end{abstract}

\section{Introduction}
It is well known that every integer flow on a graph can be
expressed as a sum of characteristic flows of circuits. By the
characteristic flow $\chi_C$ of a circuit $C$ in a graph $G$ we
mean the flow that takes values $+1$ or $-1$ on $C$ and value
$0$ anywhere else in $G$. One way of seeing this fact is to fix
an orientation of $G$, take an arbitrary spanning tree $T$ of
$G$, and express the given flow $\phi$ as
$$\phi=\sum_{x\in E(G)-E(T)}\phi(x)\cdot\chi_{T(e)}$$
where $x$ is a cotree edge, $T(x)\subseteq T+x$ is the
fundamental cycle corresponding to $x$, and the sum extends
over all cotree edges of $G$. By choosing appropriate
orientations for the cotree edges one can achieve that the
coefficients $\phi(x)$ in this sum are all non-negative. Note
that this choice induces an orientation of each fundamental
cycle $T(x)$ of $G$ and may cause that an edge $t$ of $T$
belonging to two fundamental circuits $T(y)$ and $T(z)$
receives two opposite orientations from them; equivalently, for
the same direction of $t$ one would have $\chi_{T(y)}(t)=+1$
and $\chi_{T(z)}(t)=-1$. It turns out, however, that
incompatibilities such as this one can always be avoided by
choosing the set of circuits properly. Indeed, in 1956 Tutte
\cite[6.2]{T} proved that the decomposition of $\phi$ into
characteristic flows can always be performed in such a way that
all the circuits occurring in the expression are directed
circuits with respect to a suitable fixed orientation of $G$.
The aim of this paper is to establish a similar result for
flows on signed graphs using the concept of a signed circuit
and its characteristic flow, which we explain in Section~2 and
Section~3, respectively.

Our main result reads as follows.

\begin{maintheorem}
For every integer flow  $\phi$ on a signed graph $G$ there
exists a set $\mathcal{C}$ of signed circuits of $G$ which are
consistently directed with respect to a suitable orientation of
$G$ and positive integers $n_C$ such that
$$\phi=\sum_{C\in\mathcal{C}}n_C\chi_C.$$
\end{maintheorem}

Rather than following Tutte's approach based on the use of
chain groups we prove our theorem directly by employing a
purely graph theoretical approach that inspects circuits and
their signed analogues. Of course, the special case of balanced
signed graphs will yield a proof of the original Tutte's
result.

Our paper is divided into four sections. Section~2 reviews the
basic concepts of signed graph theory with the emphasis on the
notions related to flows. In Section~3 we introduce the concept
of a characteristic flow of a signed circuit and prove our main
result. In the final section we apply the main result to the
study of the signed analogue of the shortest circuit cover
problem recently initiated in \cite{MRRS}.

\section{Signed graphs and flows}

A \textit{signed graph} is a graph in which each edge is
labelled with a sign, $+$ or $-$. An \textit{orientation}, or a
\textit{bidirection}, of a signed graph is obtained by dividing
each edge into two \textit{half-edges} and by assigning
individual orientations to them subject to the following
compatibility rule: a positive edge has one half-edge directed
from and the other half-edge directed to its end-vertex, while
a negative edge has both half-edges directed either towards or
from their respective end-vertices. Thus each edge,
irrespectively of its sign, has two possible orientations which
are opposite to each other.

Given an abelian group $A$, an \textit{$A$-flow} on a signed
graph $G$ is an assignment of an orientation and a value from
$A$ to each edge in such a way that for each vertex of $G$ the
sum of incoming values equals the sum of outgoing values
(\textit{Kirchhoff's law}). If $0\in A$ is not used as a flow
value, the flow is said to be \textit{nowhere-zero}. The
concept of a nowhere-zero $A$-flow is particularly interesting
when $A$ is the group of integers. A major problem is to
determine, for a given signed graph $G$, the smallest integer
$k\ge 2$ such that $G$ has an integer flow with values in the
set $\{\pm 1, \pm 2, \ldots, \pm (k-1)\}$; such a flow is
called a \textit{nowhere-zero $k$-flow}. In 1983, Bouchet
\cite{Bouchet} conjectured that every signed graph that admits
a nowhere-zero integer flow has a nowhere-zero $6$-flow.
Although various approximations of this conjecture have been
proved \cite{devos, RaspaudZhu, WeiTang, XuZhang,Zyka}, this
conjecture remains open.

Signed graphs that admit a nowhere-zero integer flow are called
\textit{flow-admissible}. In contrast to unsigned graphs,
describing flow-admissible is not immediate. For this purpose
we need the notions of a balance of a signed graph and that of
a signed circuit.

A circuit of a signed graph $G$ is called \textit{balanced} if
it contains even number of negative edges, otherwise it is
called \textit{unbalanced}. A signed graph itself is called
\textit{balanced} if it does not contain any unbalanced
circuit, and is called \textit{unbalanced} if it does. The
collection of all balanced circuits is the most fundamental
characteristic of a signed graph: signed graphs having the same
underlying graphs and the same sets of balanced circuits are
considered to be \textit{identical}, irrespectively of their
actual signatures.

A \textit{signed circuit} of a signed graph is a subgraph of
any of the following three types:
\begin{itemize}
\item[(1)] a balanced circuit,
\item[(2)] the union of two disjoint unbalanced circuits
    with a path that meets the circuits only at its ends,
    or
\item[(3)] the union of two unbalanced circuits that meet
    at a single vertex.
\end{itemize}
A signed circuit falling under item (2) or (3) is called an
\textit{unbalanced bicircuit}. Note that  a bicircuit from item
(3) can be regarded as as special case of the one from (2),
with the connecting path being trivial. Observe, however, that
signed circuits from items (1) or (3) admit a nowhere-zero
$2$-flow while those under item (2) admit a nowhere-zero
$3$-flow, but not a $2$-flow.

The following result is due to Bouchet \cite{Bouchet} and
reflects the fact that signed circuits are inclusion minimal
signed graphs that admit a nowhere-zero integer flow.

\begin{theorem}\label{thm:flow-adm}
A signed graph $G$ admits a nowhere-zero integer flow if and
only if each edge of $G$ belongs to a signed circuit.
\end{theorem}

\section{Decomposition into characteristic flows}

Let $\phi$ be a flow on a signed graph $G$. If we reverse the
orientation of any edge $e$ and replace the value $\phi(e)$
with $-\phi(e)$, the resulting valuation will again be a flow.
We regard this operation as a way of expressing the same flow
$\phi$ in terms of a different orientation. Thus, within a
given signature, we may choose a compatible orientation
arbitrarily. If $\phi$ is an integer flow, we can always find
an orientation for $G$ such that $\phi(e)\ge 0$ for each edge
$e$. We call this orientation a \textit{positive orientation}
of $G$ with respect to $\phi$. If $\phi$ is nowhere-zero, this
orientation is unique.

Another useful operation that preserves flows on a signed graph
is known as \textit{switching}. It consist in choosing a vertex
$v$ of $G$, reversing the orientation of each half-edge
incident with $v$, and changing the signature of $G$
accordingly. The result is an identical signed graph, because
the total sign of every circuit has not been changed, equipped
with a new compatible orientation. If $G$ carries a flow, then
the same function works as a flow for the new signature and
orientation. The same flow is thus again expressed in terms of
a different orientation and signature. By a repeated use of
switching we may turn a given signature into any other
equivalent signature~\cite{Zaslav}, keeping the flow invariant.

For a fixed orientation of $G$, the sum $\phi+\psi$ of two
flows $\phi$ and $\psi$ is defined by setting
$(\phi+\psi)(e)=\phi(e)+\phi(e)$; clearly,  $\phi+\psi$ is
again a flow. We are now interested in the reverse process of
expressing an arbitrary integer flow as a sum of suitably
chosen elementary flows. The question whether this is possible
for every flow on an arbitrary flow-admissible signed graph was
posed by Andr\'e Raspaud (personal communication) referring to
a result of Tutte \cite[6.2]{T} for unsigned graphs. In Tutte's
theorem, elementary flows are represented by characteristic
flows of circuits. Theorem~\ref{thm:flow-adm} suggests that in
the case of flows on signed graph circuits should be replaced
with signed circuits.

Consider a pair of adjacent edges $e$ and $f$ sharing a vertex
$v$ in a bidirected signed graph. We say that the walk $ef$ is
\textit{consistently directed} at $v$ if exactly one of the two
half-edges incident with $v$ is directed to~$v$. A path or a
balanced circuit is said to be \textit{consistently directed}
if all pairs of consecutive edges are consistently directed. An
unbalanced circuit is \textit{consistently directed} if it has
a single vertex, called the \textit{faulty vertex}, such that
all pairs of consecutive edges are consistently directed except
for the pair sharing the faulty vertex. Finally, an unbalanced
bicircuit is said to be \textit{consistently directed} if every
pair of adjacent edges in the bicircuit is consistently
directed except for the two edges of either circuit that share
an end-vertex of the connecting path. These vertices are the
\textit{faulty vertices} of the bicircuit. Examples of
consistently directed signed circuits are displayed in
\textbf{Fig.~\ref{fig:signedcircs}}.

\begin{figure}[htbp]
  \centerline{
     \scalebox{0.45}{
       \input{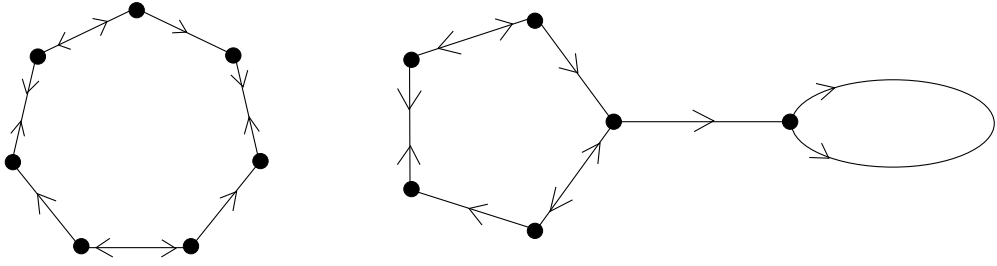}
     }
  }
\caption{Consistently directed signed circuits} \label{fig:signedcircs}
\end{figure}

It is easy to see that any signed circuit in a bidirected
signed graph can be turned into a consistently directed signed
circuit by only reversing the orientation edges. Furthermore,
if the signature is switched at some vertex, consistency of
orientation is not affected. Hence, any signed circuit may have
several different consistent orientations.

\medskip\noindent\textbf{Definition.}
Let $G$ be a signed graph and let $C$ be a signed circuit of
$G$ endowed with a consistent orientation. The
\textit{characteristic flow} of $C$ in $G$ is a function
$\chi_C\colon E(G)\to\{0, 1/2, 1\}$ defined as follows. If $C$
is a balanced circuit, we set $\chi_C(e)=1$ for each edge of
$S$ and $\chi_S(e)=0$ otherwise. If $C$ is an unbalanced
bicircuit, we set $\chi_C(e)=1$ whenever $e$ belongs to the
connecting path of~$C$, $\chi_C(e)=1/2$ whenever $e$ belongs to
a circuit of~$C$, and $\chi_C(e)=0$ otherwise.

\medskip

Note that the characteristic flow $\chi_C$ of a signed circuit
$C$ is a flow although not necessarily an integer flow. The
values of $\chi_C$ and, in fact, the values of an arbitrary
linear combination of characteristic flows of signed circuits
of a graph over integers will be contained in the cyclic
subgroup $H\le\mathbb{Q}$ generated by the element~$1/2$. The
group $H$ includes the group of integers as a subgroup of
index~$2$; the elements of $H-\mathbb{Z}$ will be called
\textit{fractional}.

Since a flow on a signed graph is invariant under the
orientation reversal and vertex-switching, the concept of a
characteristic flow applies to any signed circuit
irrespectively of its particular orientation and signature.
Furthermore, it is easy to see that, up to equivalence, the
characteristic flow of a signed circuit is uniquely determined
by the value on a single bidirected edge, which may be either
$+1$ or $-1$. However, switching at all vertices reverses the
orientation of each edge without changing the flow values. This
implies that, up to equivalence, every flow $\phi$ on a signed
graph coincides with its opposite $-\phi$. In particular, each
signed circuit has exactly one characteristic flow, up to
equivalence.

\medskip

We proceed to the main result, the decomposition theorem. For
the proof recall that the \textit{support} of a flow $\phi$,
denoted by $\supp(\phi)$, is the set of all edges $e$ for which
$\phi(e)\ne 0$.

\begin{theorem}\label{thm:decomp}
Let $\phi$ be an integer flow on a signed graph $G$. Then there
exists a set $\mathcal{C}$ of signed circuits of $G$ which are
consistently directed with respect to a positive orientation of
$G$ and positive integers $n_C$, indexed by the elements of
$\mathcal{C}$, such that
$$\phi=\sum_{C\in\mathcal{C}}n_C\chi_C.$$
\end{theorem}

\noindent\textbf{Remark.} A natural question arises whether for
this theorem to be true the fractional values in the definition
of a characteristic flow are really necessary. The answer is,
unfortunately, `yes'. To see this, let us consider the signed
graph $G$ consisting of two vertices joined by a pair of
positive parallel edges with a negative loop attached at either
vertex. It is easy to see that $G$ admits a nowhere-zero
$2$-flow, but any decomposition of this flow into the sum of
characteristic flows will contain characteristic flows of two
distinct unbalanced bicircuits, each with coefficient $1$. The
reader can easily extend this example into an infinite series
of similar examples where a nowhere-zero $2$-flow only
decomposes into the sum of characteristic flows of two distinct
unbalanced bicircuits, each with coefficient~$1$.

\medskip\noindent\textit{Proof of Theorem~\ref{thm:decomp}.}
Throughout the proof we keep fixed a positive orientation of
$G$ with respect to $\phi$. The proof is trivial if $\phi=0$,
so we may assume that $\supp(\phi)\ne\emptyset$. If
$\supp(\phi)$ contains a consistently directed balanced circuit
$B$, we form the flow $\phi-\chi_B$ which is again an integer
flow on~$G$. We repeat the procedure with the flow
$\phi-\chi_B$ and continue as long as the support of the
current flow contains a consistently directed balanced circuit.
Eventually we obtain a set $\mathcal{B}$ of consistently
directed balanced circuits and the flow
$\phi_1=\phi-\sum_{B\in\mathcal{B}}\chi_B$. If $\phi_1=0$, then
$\phi=\sum_{B\in\mathcal{B}}\chi_B$, and the required
expression for $\phi$ follows immediately. Otherwise
$\supp(\phi_1)$ is nonempty and induces a subgraph $G_1$ which
carries a non-null integer flow $\phi_1$. Clearly, $G_1$
contains no consistently directed balanced circuit.

Our next aim is to show that $G_1$ contains a consistently
directed unbalanced bicircuit. To this end, we first identify a
consistently directed unbalanced circuit in $G_1$. We pick an
arbitrary vertex $u_1$ of $G_1$ and choose an edge $e_1$ that
leaves $u_1$; since $G_1$ has a positive orientation such an
edge always exists. Let $u_2$ be the other end of $e_1$. At
$u_2$, there must be an edge $e_2$ such that the walk $e_1e_2$
is consistent at $u_2$. We continue in the same manner until we
reach a vertex previously visited, say $u_i$. The segment
between the two occurrences of $u_i$ is clearly a circuit $D$
of~$G_1$. By the construction, $D$ is consistent everywhere
except possibly $u_i$. Since $G_1$ contains no consistently
directed balanced circuit, $D$ is a consistently directed
unbalanced circuit and $u_i$ is its faulty vertex.

Next we show that $D$ is contained in a directed unbalanced
bicircuit. The edges of $D$ incident with $u_i$ are either both
directed to $u_i$ or both from $u_i$. Since $G_1$ has a
positive orientation, there is an edge $f_1$ incident with
$u_i$ which is consistent at $u_i$ with any of the two edges of
$D$ incident with $u_i$. Set $v_1=u_i$ and let $v_2$ denote the
other end of~$f_1$. At $v_2$, there must be an edge $f_2$ such
that the walk $f_1f_2$ is consistent at $v_2$. Again, we
continue similarly until we reach a vertex $v_j$ that either
belongs to $D$ or coincides with a vertex $v_m$ with $m<j$.
Observe that $v_j$ does not lie on $D-v_1$, for if it does, we
can divide $D$ into two $v_j$-$v_1$-segments $D_1$ and $D_2$
exactly one of which is consistently directed with $f_{j-1}$
at~$v_j$. But then one of $f_1f_2\ldots f_{j-1}D_1$ or
$f_1f_2\ldots f_{j-1}D_2$ is a consistently directed balanced
circuit, a contradiction (see \textbf{Fig.~\ref{fig:pripad}}
for illustration). It follows that $v_j=v_m$ for some $m<j$. In
this case, however, $D'=f_mf_{m+1}\ldots f_{j-1}$ is a
consistently directed unbalanced circuit which together with
$D$ and the path $f_1f_2\ldots f_{m-1}$ forms a consistently
directed unbalanced bicircuit $U$. It may happen that $u_m=v_i$
in which case the connecting path of the bicircuit is trivial.

\begin{figure}[htbp]
  \centerline{
     \scalebox{0.45}{
       \input{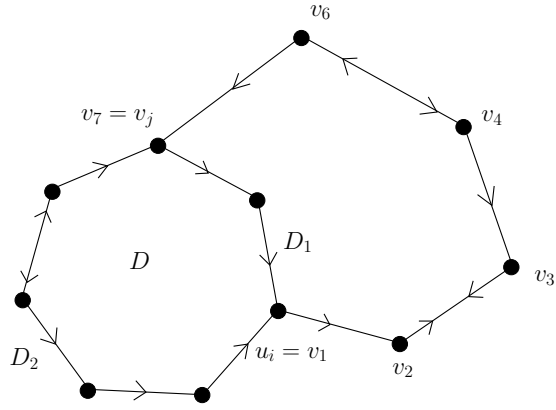}
     }
  }
\caption{Contradiction in constructing an unbalanced bicircuit} \label{fig:pripad}
\end{figure}

We now take the characteristic flow $\chi_U$, construct the
flow $\phi_2=\phi_1-\chi_U$, and set $G_2=\supp(\phi_2)$. Note
that $\phi_2$ is not an integer flow anymore, but its
fractional values are confined to two edge-disjoint
consistently directed unbalanced circuits of $G_2$.
Furthermore, $G_2$ contains no consistently directed balanced
circuit because such a circuit would also be contained in
$G_1$, which is impossible. The next step is to show that $G_2$
contains a consistently directed unbalanced bicircuit $U'$ such
that the support of the flow $\phi_3=\phi_2-\chi_{U'}$ has
either none or exactly two edge-disjoint consistently directed
unbalanced circuits carrying fractional values. Since neither
$\supp(\phi_2)$ nor $\supp(\phi_3)$ contain consistently
directed balanced circuits, repeating this procedure will
necessarily terminate with a set $\mathcal{U}$ of consistently
directed unbalanced bicircuits such that
$\phi_1-\sum_{U\in\mathcal{U}}\chi_U=0$, implying that
$$\phi=\sum_{B\in\mathcal{B}}\chi_B+\sum_{U\in\mathcal{U}}\chi_U.$$
Since the latter expression immediately yields the statement of
the theorem, all that remains is to describe a procedure that
starts with a flow $\psi$ on $G$ such that
\begin{itemize}
\item[(1)] $\supp(\psi)$ contains no consistently directed
    balanced circuit, and
\item[(2)] fractional values of $\psi$ occur in exactly two
    edge-disjoint unbalanced circuits of $\supp(\psi)$.
\end{itemize}
and constructs an unbalanced bicircuit $U'\subseteq
\supp(\psi)$ such that in $\supp(\psi-\chi_{U'})$ there are
either none or exactly two edge-disjoint consistently directed
unbalanced circuits carrying fractional values.

Let $\psi$ be a flow on $G$ such that $\supp(\psi)$ satisfies
(1) and (2) stated above. Let $d$ and $d'$ be the faulty
vertices of $D$ and $D'$, respectively. To construct an
unbalanced bicircuit $U'$ let us take the circuit $D$, choose
an edge $e$ incident with $d$ which is consistently directed
with either of the two adjacent edges of $D$ and proceed by
successively constructing a consistently directed trail~$T$
until we either reach a previously encountered vertex of $T$ or
a vertex of $D\cup D'$. Let $t$ be the terminal vertex of $T$.

First observe that $t$ cannot belong to $D-d$. Indeed,
otherwise we could split $D$ into two $t$-$d$-segments $D_1$
and $D_2$ producing circuits $TD_1$ and $TD_2$ one of which
would be a consistently directed balanced circuit in
$\supp(\psi)$, contradicting (2).

There remain four possibilities for the position of $t$ to
consider.

\medskip\noindent
\textbf{Case 1.} \textit{The vertex $t$ coincides with a
previously encountered vertex of $T$, possibly $t=d$.} It
follows that the portion of $T$ between the two occurrences of
$t$ forms a consistently directed unbalanced circuit, say
$D''$, and thus $D\cup T$ forms a consistently directed
unbalanced bicircuit, the sought $U'$. Indeed, $D'$ and $D''$
are edge-disjoint and $(D\cup D')\cap (D\cup D'')=D$, hence the
support of the flow $\psi-\chi_{U'}$ again contains precisely
two directed unbalanced circuits carrying fractional values,
namely $D'$ and $D''$.

\medskip\noindent
\textbf{Case 2.} \textit{The vertex $t$ belongs to $D'-d'$.}
Let $D'_1$ and $D'_2$ denote the two $t$-$d'$-segments of~$D'$.
Then exactly one of  $TD'_1$ and $TD'_2$, say $TD'_1$, is a
consistently directed $d$-$d'$-path. Starting from $d'$
construct a consistently directed trail $T'$ whose first edge
is consistently directed with either of the two adjacent edges
of $D'$ and continue until we either reach a previously visited
vertex of $T'$ or a vertex of $D\cup D'\cup T$. Let $t'$ be the
first such vertex.

Observe that $t'$ belongs neither to $D-d$ not to $D'-d'$.
Otherwise, in the former case $D$ would contain a
$t'$-$d$-segment $S$ such that $STD'_1T'$ is a consistently
directed balanced circuit in $\supp(\psi)$, and similarly in
the latter case $D'$ would contain a $t'$-$d'$-segment $S'$
such that $S'T'$ is a a consistently directed balanced circuit
in $\supp(\psi)$. In both cases we would get a contradiction.

There remain two possibilities for the position of $t'$.

\medskip\noindent
\textbf{Subcase 2.1.}  \textit{The vertex $t'$ coincides with a
previously encountered vertex of $T'$, possibly $t'=d'$.} The
portion of $T'$ between the two occurrences of $t'$ forms a
consistently directed unbalanced circuit, say $D''$, and hence
$D'\cup T'$ is a consistently directed unbalanced bicircuit. If
we set $U'=D'\cup T'$, then the support of the flow
$\psi-\chi_{U'}$ again contains precisely two directed
unbalanced circuits carrying fractional values, namely $D$ and
$D''$.

\medskip\noindent
\textbf{Subcase 2.2.} \textit{The vertex $t'$  belongs to $T$,
possibly $t'=d$.} It is obvious that $t'\ne t$ because
otherwise $t'$ would lie in $D'-d'$, which we have shown to be
impossible. On the other hand, $t'$ may coincide with~$d$. The
vertex $t'$ splits the path $T$ into two segments, a
$d$-$t'$-segment $W_1$, which may be trivial, and a
$t'$-$t$-segment $W_2$, which is nontrivial. Consider the last
edge $g$ of $T'$ and the first edge $h$ of the segment $W_2$.
If $g$ was consistently directed with $h$ at~$t'$, then
$W_2D_1'T'$ would be a consistently directed balanced circuit
within $\supp(\psi)$, which is impossible. Thus $g$ is not
consistently directed with $h$ at $t'$. Since $h$ is
consistently directed at $t'$ with the last edge of $W_1$,
provided that $t'\ne d$, or with both edges of $D$ incident
with~$d$, provided that $t'=d$, it follows that $W_1\cup T'$ is
a consistently directed path whose ends $d$ and $d'$ are the
only faulty vertices of $D\cup W_1\cup T'\cup D'$. Thus  $D\cup
W_1\cup T'\cup D'$ is a consistently directed unbalanced
bicircuit, the sought $U'$. It is now easy to see that
$\psi-\chi_{U'}$ is an integer flow whose support does not
contain any consistently balanced circuit.

\medskip\noindent
\textbf{Case 3.} \textit{The vertex $t$ coincides with $d'$ and
the terminal edge of $T$ is consistently directed at $d'$ with
either of the two adjacent edges of $D'$.} In this case $D\cup
T\cup D'$ is a consistently directed unbalanced bicircuit. We
can set $U'=D\cup T\cup D'$ and observe that $\psi-\chi_{U'}$
is an integer flow whose support contains no consistently
directed balanced circuit. Again, the required conclusion
holds.

\medskip\noindent
\textbf{Case 4.} \textit{The vertex $t$ coincides with $d'$ but
the terminal edge of $T$ is not consistently directed at $d'$
with the two adjacent edges of $D'$.} By the Kirchhoff law,
there exists an edge in $\supp(\psi)$ incident with $d'$ which
is consistently directed at $d'$ with both adjacent edges of
$D'$. Hence, starting from this edge we can again construct a
consistently directed trail $T'$ which terminates by reaching
either a previously visited vertex of $T'$ or a vertex of
$D\cup D'\cup T$. As in Case~2, the terminal vertex $t'$ cannot
lie in $(D-d)\cup(D'-d')$ for otherwise we could find a
consistently balanced circuit in $\supp(\psi)$,
contradicting~(2). There remain two possibilities for the
position of $t'$ which are completely analogous to Subcases~2.1
and 2.2, and are therefore left to the reader.

\medskip

As we have seen, in each case the procedure can either be
continued or will terminate with the zero flow. The proof is
complete. \hfill $\Box$

\medskip

Observe that the proof of Theorem~\ref{thm:decomp} makes no use
of the characterisation of flow-admissible graphs given in
Theorem~\ref{thm:flow-adm}. Just on the contrary,
Theorem~\ref{thm:flow-adm} easily follows from
Theorem~\ref{thm:decomp}.

\begin{corollary}
A signed graph $G$ admits a nowhere-zero integer flow if and
only if each edge of $G$ belongs to a balanced circuit or an
unbalanced bicircuit.
\end{corollary}

\begin{proof}
The forward implication is an immediate consequence of
Theorem~\ref{thm:decomp}. For the converse,  let
$\mathcal{C}=\{C_1,C_2, \ldots, C_r\}$ be a set of signed
circuits such that each edge of $G$ belongs to a member of
$\mathcal{C}$. We first fix an arbitrary orientation of $G$;
note that the elements of $\mathcal{C}$ need not be
consistently directed with respect to this orientation. Now we
can define the function $\psi\colon E(G)\to\mathbb{Z}$ by
setting
$$\psi=\sum_{i=1}^{r}2^{2i-1}\chi_{C_i}.$$ It is easy to see that
$\psi$ is indeed a nowhere-zero integer flow on $G$.
\end{proof}

\section{Application to signed circuit covers}

A circuit cover of an unsigned graph is  a collection of
circuits such that each edge of the graph belongs to at least
one of the circuits. It is a standard problem to find, for a
given graph, a circuit cover of minimum total length. It has
been conjectured that every bridgeless graph $G$ has a circuit
cover of length at most $7|E(G)|/5$ (Jaeger, private
communication; independently \cite{AT}), but the best current
general bound is $5|E(G)|/3$ (see \cite{AT,BJJ}). The
$7/5$-conjecture is particularly interesting for its
relationship to other prominent conjectures in graph theory.
For instance, its validity is implied by the Petersen flow
conjecture (alternatively known as the Petersen colouring
conjecture) of Jaeger \cite[Section~7]{jaeger}, while the
conjecture itself implies the celebrated cycle double cover
conjecture (see Raspaud \cite{raspaud-diss} and Jamshy and
Tarsi \cite{jamshy}).

A natural analogue of a circuit cover for signed graphs is the
concept of a \textit{signed circuit cover} introduced in
\cite{MRRS}. It is a collection $\mathcal{C}$ of signed
circuits of a signed graph $G$ such that each edge of $G$ is
contained in at least one member of $\mathcal{C}$. In
\cite{MRRS} it was shown that every flow-admissible signed
graph $G$ has a signed circuit cover of total length at most
$11|E(G)|$.

We now apply our Theorem~\ref{thm:decomp} to the shortest
signed circuit problem. If a signed graph $G$ admits a
nowhere-zero integer $k$-flow $\phi$, we can decompose it into
the sum $\sum_{C\in\mathcal{C}}n_C\chi_C$ of characteristic
flows guaranteed by Theorem~\ref{thm:decomp}. Obviously, the
set $\mathcal{C}$ provides a signed circuit cover of $G$. This
cover yields the following bounds.

\begin{corollary}
If a signed graph $G$ has a nowhere-zero flow $\phi$, then $G$
has a signed circuit cover such that each edge $e$ belongs to
at most $2|\phi(e)|$ signed circuits.
\end{corollary}

\begin{corollary}
If a signed graph $G$ admits a nowhere-zero $k$-flow, then it
has a signed circuit cover of total length at most
$2(k-1)|E(G)|$.
\end{corollary}

Observe that if Bouchet's $6$-flow conjecture \cite{Bouchet} is
true, then the previous corollary implies that every
flow-admissible signed graph $G$ has a signed circuit cover of
total length at most $10|E(G)|$.

\bigskip{}\bigskip

\noindent\textbf{Acknowledgements.}  We acknowledge partial
support from the grants APVV-0223-10 and VEGA 1/1005/12.

\end{document}